\documentclass[11pt,a4paper]{article}
\usepackage[latin1]{inputenc}
\usepackage{amsmath}
\usepackage{amsthm}
\usepackage{amsfonts}
\usepackage{amsfonts,amsthm,latexsym,amsmath,amssymb,amscd,epsfig,psfrag,enumerate}
\usepackage{graphics,graphicx, bezier, float, color, hyperref}
\usepackage{amssymb,url}
\usepackage{multienum}
\usepackage[table]{xcolor}
\usepackage{multicol,multirow}
\usepackage{graphicx}
\usepackage{fancyvrb}
\usepackage{parskip}
\usepackage[toc,page]{appendix}
\usepackage[top=2.7cm, bottom=2.7cm, left=1.5cm, right=1.5cm]{geometry}
\usepackage{xcolor}
\usepackage{blkarray}
\newtheorem{theorem}{Theorem}[section]
\newtheorem{lemma}{Lemma}[section]
\newtheorem{cor}{Corollary}[section]

\usepackage[none]{hyphenat}[section]

\numberwithin{equation}{section}
\numberwithin{table}{section}
\numberwithin{figure}{section}
\title{$k$-Fibonacci and $k$-Lucas numbers with the H\"older inequality}
\author{Herbert Batte$^{1}$ and Prosper Kaggwa $^{2,*}$}
\date{}
\begin{document}
	\maketitle
\begin{abstract}
	Dujella, Jak\v seti\'c and Pe\v cari\'c recently established in \cite{DJP}, a chain of power-sum inequalities, built from H\"older's and Cauchy's inequalities and their converse forms, and applied it to the Fibonacci sequence using the identities $\sum_{i=1}^nF_i^2=F_nF_{n+1}$ and $\sum_{i=1}^nF_iF_{i+1}=F_{n+1}^2-\frac{1+(-1)^n}{2}$. We show that this machinery applies uniformly across the one-parameter family of $k$-Fibonacci and $k$-Lucas numbers of Falc\'on and Plaza, via the generalized identities
	\begin{align*}
	\sum_{i=1}^nL_{k,i}^2=\frac{L_{k,n}L_{k,n+1}-2k}{k},\qquad
	\sum_{i=1}^nF_{k,i}^2=\frac{F_{k,n}F_{k,n+1}}{k},
	\end{align*}
	and
	\begin{align*}
	\sum_{i=1}^nL_{k,i}L_{k,i+1}=\frac{L_{k,n+1}^2}{k}-k+\bigl((-1)^n-1\bigr)\left(\frac{2}{k}+\frac{k}{2}\right),
\end{align*}
which recover the classical Fibonacci and Lucas identities at $k=1$. We further exploit the cross-identity $F_{k,i}L_{k,i}=F_{k,2i}$, valid for every $k\ge1$, to obtain a Cauchy-Schwarz-type inequality linking the two families that has no counterpart in the Fibonacci-only setting. At $k=1$, this specializes to a fully explicit elementary inequality between ordinary Fibonacci and Lucas numbers, alongside the corresponding H\"older and Cauchy-conversion refinements we obtain for ordinary Lucas numbers.
\end{abstract}
	
	
	
	\thanks{$ ^{*} $ Corresponding author}
	
	\section{Introduction}

\subsection{Background}\label{sec:background}

The Fibonacci sequence
\begin{align*}
	F_0=0,\quad F_1=1,\quad F_{n}=F_{n-1}+F_{n-2}, \qquad \text{for}\qquad n\ge 2,
\end{align*} 
and its companion, the Lucas sequence 
\begin{align*}
	L_0=2,\quad L_1=1,\quad L_{n}=L_{n-1}+L_{n-2}, \qquad \text{for}\qquad n\ge 2,
\end{align*}  
are among the most studied objects in elementary Number theory, and a substantial literature is devoted to inequalities involving them. Among the classical tools used to produce such inequalities are the H\"older and Cauchy-Schwarz inequalities. Given $p>1$ and its conjugate exponent $q=p/(p-1)$, and positive reals $x_1,\dots,x_n,\,y_1,\dots,y_n$, H\"older's inequality states that
\begin{align*}
\sum_{i=1}^n x_iy_i\le\Bigl(\sum_{i=1}^n x_i^p\Bigr)^{1/p}\Bigl(\sum_{i=1}^n y_i^q\Bigr)^{1/q},
\end{align*} 
with the Cauchy-Schwarz inequality as the special case $p=q=2$. Less familiar, but equally classical, are the \emph{converse} forms of these inequalities (due to Diaz and Metcalf, and to P\'olya and Szeg\H o), which bound the same quantities from below (or above, depending on the direction) once the ratios $x_i/y_i^{q/p}$ are known to lie in a fixed interval $[m,M]$.

To organize inequalities of this type, it is convenient to work with the power sum
\begin{align*}
S_n^{[\alpha]}(x)=\sum_{i=1}^n x_i^{\alpha},\qquad \alpha\in\mathbb{R},\ x=(x_1,\dots,x_n)\in\mathbb{R}_+^n,
\end{align*} 
following Dujella, Jak\v seti\'c and Pe\v cari\'c \cite{DJP}. Chaining H\"older's inequality with Jensen's inequality and the monotonicity of the power means, they obtained, for $x_i\ge1$ and $\alpha=u/p+v/q\ge\beta>0$,
\begin{align*}
\bigl(S_n^{[u]}(x)\bigr)^{1/p}\bigl(S_n^{[v]}(x)\bigr)^{1/q}\ \ge\ S_n^{[\alpha]}(x)\ \ge\ \frac{1}{n^{\alpha/\beta-1}}\bigl(S_n^{[\beta]}(x)\bigr)^{\alpha/\beta}\ \ge\ S_n^{[\beta]}(x)\ \ge\ \bigl(S_n^{[\alpha]}(x)\bigr)^{\beta/\alpha},
\end{align*} 
together with companion chains for $0<\alpha<\beta$ and for the reversed H\"older inequality ($0<p<1$), and sharper, constant-tracking versions of the same chains obtained from the converse H\"older and converse Cauchy inequalities. The point of \cite{DJP} is that this machinery is entirely agnostic to what the sequence $x$ is. Any positive sequence for which $S_n^{[\beta]}(x)$ has a closed form at some fixed $\beta$ immediately yields an inequality chain for $S_n^{[\alpha]}(x)$ at every other exponent $\alpha$. Applying it to the Fibonacci sequence via the classical identities
\begin{equation}\label{eq:fib-identities}
	\sum_{i=1}^nF_i^2=F_nF_{n+1},\qquad \sum_{i=1}^nF_iF_{i+1}=F_{n+1}^2-\frac{1+(-1)^n}{2},
\end{equation}
they recovered, and refined, an earlier inequality of Alzer and Luca \cite{AlzerLuca},
\begin{align*}
\sum_{j=1}^nF_j^{r}\sum_{j=1}^nF_j^{s}\ge\bigl(F_nF_{n+1}\bigr)^2,\qquad r+s\ge4,
\end{align*} 
and produced several further chains of inequalities for sums of powers of Fibonacci numbers.

Since the underlying machinery of \cite{DJP} never uses any special property of the Fibonacci sequence beyond positivity and the existence of the two closed-form identities above, it is natural to ask how these results behave when the Fibonacci sequence is replaced by a genuine \emph{family} of sequences depending on a parameter, rather than by a single fixed sequence such as the Lucas numbers. This is the question we take up here, using the $k$-Fibonacci and $k$-Lucas numbers of Falc\'on and Plaza as the relevant family.

For a positive integer $k$, the \emph{$k$-Fibonacci numbers} and \emph{$k$-Lucas numbers} are defined by the second-order linear recurrences
\begin{align*}
F_{k,0}=0,\quad F_{k,1}=1,\quad F_{k,n+1}=kF_{k,n}+F_{k,n-1}\quad (n\ge 1),
\end{align*} 
and
\begin{align*}
L_{k,0}=2,\quad L_{k,1}=k,\quad L_{k,n+1}=kL_{k,n}+L_{k,n-1}\quad (n\ge 1),
\end{align*} 
following Falc\'on and Plaza \cite{falconplaza} and Falc\'on \cite{falconlucas}. Writing
\begin{align*}
\sigma_1=\frac{k+\sqrt{k^2+4}}{2},\qquad \sigma_2=\frac{k-\sqrt{k^2+4}}{2},
\end{align*} 
so that $\sigma_1+\sigma_2=k$ and $\sigma_1\sigma_2=-1$, the sequences admit the Binet-type formulas
\begin{align*}
F_{k,n}=\frac{\sigma_1^n-\sigma_2^n}{\sigma_1-\sigma_2},\qquad L_{k,n}=\sigma_1^n+\sigma_2^n.
\end{align*} 
The case $k=1$ recovers the classical Fibonacci numbers $F_n$ and Lucas numbers $L_n$; the case $k=2$ recovers the Pell and Pell-Lucas numbers.

\subsection{Main results}

We first establish the identities that will play the role of \eqref{eq:fib-identities} above for the $k$-Fibonacci and $k$-Lucas sequences. Throughout, set
\begin{align*}
D_k(n):=\frac{L_{k,n}L_{k,n+1}-2k}{k},\qquad P_k(n):=\frac{L_{k,n+1}^2}{k}-k+\bigl((-1)^n-1\bigr)\left(\frac{2}{k}+\frac{k}{2}\right).
\end{align*} 

\begin{theorem}\label{thm:identities}
	For every positive integer $k$ and every $n\ge1$,
	\begin{align*}
	\sum_{i=1}^nL_{k,i}^2=D_k(n),\qquad
	\sum_{i=1}^nF_{k,i}^2=\frac{F_{k,n}F_{k,n+1}}{k},\qquad
	\sum_{i=1}^nL_{k,i}L_{k,i+1}=P_k(n),
\end{align*} 
	and for every $i\ge1$,
	\[
	F_{k,i}L_{k,i}=F_{k,2i}.
	\]
	All four identities reduce, at $k=1$, to the corresponding classical Fibonacci or Lucas identities.
\end{theorem}

Feeding the first three identities of Theorem \ref{thm:identities}, together with the general power-sum theorems of \cite{DJP}, into the sequences $x_i=L_{k,i}$ (with $\beta=2$) and $x_i=L_{k,i}L_{k,i+1}$ (with $\alpha=1$) yields the following $k$-dependent inequality chains.

\medskip

\begin{theorem}[$k$-Lucas H\"older sandwich]\label{thm:main1}
	Let $k\ge1$ be an integer, let $n\ge1$, and let $u,v\in\mathbb{R}$.
	\begin{enumerate}[(i)]
		\item If $p>1$, $q=p/(p-1)$, and $\alpha=u/p+v/q\ge2$, then
		\begin{align*}
			\Bigl(\sum_{i=1}^nL_{k,i}^u\Bigr)^{1/p}\Bigl(\sum_{i=1}^nL_{k,i}^v\Bigr)^{1/q}\ge\sum_{i=1}^nL_{k,i}^\alpha\ge\frac{1}{n^{\alpha/2-1}}D_k(n)^{\alpha/2}\ge D_k(n)\ge\Bigl(\sum_{i=1}^nL_{k,i}^\alpha\Bigr)^{2/\alpha}.
		\end{align*}
		\item If $p>1$, $q=p/(p-1)$, and $0\le\alpha=u/p+v/q<2$, then
		\begin{align*}
			\Bigl(\sum_{i=1}^nL_{k,i}^u\Bigr)^{1/p}\Bigl(\sum_{i=1}^nL_{k,i}^v\Bigr)^{1/q}\ge\sum_{i=1}^nL_{k,i}^\alpha\ge D_k(n)^{\alpha/2}.
		\end{align*}
		\item If $0<p<1$, $q=p/(p-1)$, and $\alpha=u/p+v/q\ge2$, then
		\begin{align*}
			\Bigl(\sum_{i=1}^nL_{k,i}^u\Bigr)^{1/p}\Bigl(\sum_{i=1}^nL_{k,i}^v\Bigr)^{1/q}\le\sum_{i=1}^nL_{k,i}^\alpha\le D_k(n)^{\alpha/2}.
		\end{align*}
		\item If $0<p<1$, $q=p/(p-1)$, and $0<\alpha=u/p+v/q<2$, then
		\begin{align*}
			\Bigl(\sum_{i=1}^nL_{k,i}^u\Bigr)^{1/p}\Bigl(\sum_{i=1}^nL_{k,i}^v\Bigr)^{1/q}\le\sum_{i=1}^nL_{k,i}^\alpha\le\frac{1}{n^{\alpha/2-1}}D_k(n)^{\alpha/2}\le D_k(n)\le\Bigl(\sum_{i=1}^nL_{k,i}^\alpha\Bigr)^{2/\alpha}.
		\end{align*}
		\end{enumerate}
		At $\alpha=0$ the terminal exponent $2/\alpha$ in part (iv) is undefined; there $S_n^{[0]}(x)=n$ for any positive sequence $x$, and the first three members of the chain in part (iv) remain valid with $D_k(n)^0=1$, but the last term must be omitted.
	\end{theorem}
\medskip
\begin{theorem}[Converse-H\"older refinement]\label{thm:main2}
	Let $k\ge1$ be an integer, $p>1$, $q=p/(p-1)$, $n\ge2$, and let $u,v\in\mathbb{R}$ with $u/p+v/q=2$ and $u\ne v$. Set
\begin{align*}
m=\min\bigl\{L_{k,1}^{(u-v)/p},\,L_{k,n}^{(u-v)/p}\bigr\},\qquad M=\max\bigl\{L_{k,1}^{(u-v)/p},\,L_{k,n}^{(u-v)/p}\bigr\}.
\end{align*}
By Lemma \ref{lem:monotone} and $u\ne v$, $0<m<M$. Then
\begin{align*}
(M-m)\sum_{i=1}^nL_{k,i}^u+(mM^p-Mm^p)\sum_{i=1}^nL_{k,i}^v\le(M^p-m^p)\,D_k(n),
\end{align*}
and
\begin{align*}
\Bigl(\sum_{i=1}^nL_{k,i}^u\Bigr)^{1/p}\Bigl(\sum_{i=1}^nL_{k,i}^v\Bigr)^{1/q}\le\lambda\,D_k(n),
\end{align*}
where
$\lambda=\bigl|M^p-m^p\bigr|\,\bigl|p(M-m)\bigr|^{-1/p}\,\bigl|q(mM^p-Mm^p)\bigr|^{-1/q}$.
If instead $0<p<1$, both inequalities reverse.
\end{theorem}
\medskip
\begin{theorem}[Cauchy-conversion refinement]\label{thm:main3}
	Let $k\ge1$ be an integer, let $n\ge1$, and let $u,v\in\mathbb{R}$ with $u+v=2$, and set $x_i=L_{k,i}L_{k,i+1}$,
	\begin{align*}
	m_1&=\min\{x_1^{u/2},x_n^{u/2}\},\quad M_1=\max\{x_1^{u/2},x_n^{u/2}\},\\
	m_2&=\min\{x_1^{v/2},x_n^{v/2}\},\quad M_2=\max\{x_1^{v/2},x_n^{v/2}\}.
	\end{align*}
	Then
	\begin{align*}
	1\ \le\ \frac{\bigl(\sum_{i=1}^nx_i^u\bigr)\bigl(\sum_{i=1}^nx_i^v\bigr)}{P_k(n)^2}\ &\le\ \frac14\left(\sqrt{\frac{M_1M_2}{m_1m_2}}+\sqrt{\frac{m_1m_2}{M_1M_2}}\right)^2,
	\\
	\frac{\sum_{i=1}^nx_i^u}{P_k(n)}-\frac{P_k(n)}{\sum_{i=1}^nx_i^v}\ &\le\ \left(\sqrt{\frac{M_1}{m_2}}-\sqrt{\frac{m_1}{M_2}}\right)^2,
	\\
	\Bigl(\sum_{i=1}^nx_i^u\Bigr)\Bigl(\sum_{i=1}^nx_i^v\Bigr)-P_k(n)^2\ &\le\ \frac{n^2}{4}(M_1M_2-m_1m_2)^2,\qquad\text{and}
	\\		\sum_{i=1}^nx_i^v+\frac{m_2M_2}{M_1m_1}\sum_{i=1}^nx_i^u\ &\le\ \left(\frac{M_2}{m_1}+\frac{m_2}{M_1}\right)P_k(n).
\end{align*}
\end{theorem}

The fourth identity of Theorem~\ref{thm:identities} has no analogue in the Fibonacci-only setting of \cite{DJP}, since it relates the two families $F_{k,\cdot}$ and $L_{k,\cdot}$ rather than closing either one individually. Combined with the Cauchy-Schwarz inequality, it yields the following cross-family bound.

\medskip

\begin{theorem}[Cross-family Cauchy-Schwarz inequality]\label{thm:main4}
	For every $k\ge1$ and $n\ge1$,
	\begin{align*}
	\Bigl(\sum_{i=1}^nF_{k,2i}\Bigr)^2\ \le\ \frac{F_{k,n}F_{k,n+1}}{k}\cdot D_k(n),
	\end{align*}
	with equality if and only if $n=1$.
\end{theorem}

Theorems \ref{thm:main1}-\ref{thm:main4} are stated in full generality in $k$; setting $k=1$ recovers, respectively, the ordinary-Lucas analogue of the Alzer-Luca-type refinement (\cite[Theorem 10]{DJP}), the ordinary-Lucas analogues of \cite[Theorems 11-12]{DJP}, and a Fibonacci-Lucas cross bound not present in \cite{DJP}.

\medskip

The paper is organized as follows. Section \ref{sec:prelim} recalls the required preliminaries: further identities for $k$-Fibonacci and $k$-Lucas numbers, and a precise restatement of the general power-sum theorems of \cite{DJP} that Theorems \ref{thm:main1} to \ref{thm:main3} rely on. Section \ref{sect3} proves Theorem \ref{thm:identities} and derives Theorems \ref{thm:main1} to \ref{thm:main3} from it. Section \ref{sect4} proves Theorem \ref{thm:main4} and records further $k$-Fibonacci/$k$-Lucas power-sum identities amenable to the same treatment. Section \ref{sect5} collects concluding remarks, including the specializations to ordinary Fibonacci and Lucas numbers at $k=1$.

\section{Preliminaries}\label{sec:prelim}

We collect here the identities and general theorems on which the proofs of Section \ref{sect3} rest. Section \ref{sec:prelim-klucas} records the $k$-Fibonacci/$k$-Lucas facts needed to prove Theorem \ref{thm:identities} and to justify the monotonicity used in Theorems~\ref{thm:main2} and \ref{thm:main3}. Section \ref{sec:prelim-djp} restates, without proof, the general power-sum theorems of \cite{DJP} from which Theorems \ref{thm:main1} to \ref{thm:main3} are derived.

\subsection{Identities for $k$-Fibonacci and $k$-Lucas numbers}\label{sec:prelim-klucas}

Throughout, $k$ denotes a fixed positive integer, and $\sigma_1,\sigma_2$ are as in Section~\ref{sec:background}, so that $\sigma_1+\sigma_2=k$, $\sigma_1\sigma_2=-1$, and $\sigma_1>1>0>\sigma_2>-1$.

\begin{lemma}\label{lem:step}
	For every $i\ge1$,
	\begin{align*}
	L_{k,i+1}-L_{k,i-1}=kL_{k,i},\qquad F_{k,i+1}-F_{k,i-1}=kF_{k,i}.
	\end{align*}
\end{lemma}
\begin{proof}
	Immediate from the recurrence $L_{k,i+1}=kL_{k,i}+L_{k,i-1}$ (respectively $F_{k,i+1}=kF_{k,i}+F_{k,i-1}$), rearranged.
\end{proof}

\begin{lemma}\label{lem:squares}
	For every $m\ge0$,
	\begin{align*}
	L_{k,m}^2=L_{k,2m}+2(-1)^m.
\end{align*}
\end{lemma}
\begin{proof}
	By the Binet formula $L_{k,m}=\sigma_1^m+\sigma_2^m$,
	\begin{align*}
	L_{k,m}^2=\sigma_1^{2m}+\sigma_2^{2m}+2(\sigma_1\sigma_2)^m=L_{k,2m}+2(-1)^m,
\end{align*}
	using $\sigma_1\sigma_2=-1$.
\end{proof}

\begin{lemma}\label{lem:product}
	For every $i\ge0$,
	\[
	L_{k,i}L_{k,i+1}=L_{k,2i+1}+k(-1)^i.
	\]
\end{lemma}
\begin{proof}
	By the Binet formulas,
	\[
	L_{k,i}L_{k,i+1}=(\sigma_1^i+\sigma_2^i)(\sigma_1^{i+1}+\sigma_2^{i+1})
	=\sigma_1^{2i+1}+\sigma_2^{2i+1}+\sigma_1^i\sigma_2^i(\sigma_1+\sigma_2)
	=L_{k,2i+1}+k(\sigma_1\sigma_2)^i,
	\]
	and $\sigma_1\sigma_2=-1$ gives the claim.
\end{proof}

\begin{lemma}[Monotonicity]\label{lem:monotone}
	For every $i\ge1$, $L_{k,i+1}>L_{k,i}$. Consequently the sequence $\bigl(L_{k,i}L_{k,i+1}\bigr)_{i\ge1}$ is also strictly increasing.
\end{lemma}
\begin{proof}
	By Lemma~\ref{lem:step}, $$L_{k,i+1}-L_{k,i}=(k-1)L_{k,i}+L_{k,i-1},$$ for $i\ge1$. Since $L_{k,j}>0$ for all $j\ge0$ (immediate by induction from $L_{k,0}=2$, $L_{k,1}=k\ge1$, and the recurrence) and $k\ge1$, both terms on the right are nonnegative, and $L_{k,i-1}>0$; hence $L_{k,i+1}-L_{k,i}\ge L_{k,i-1}>0$. (The statement fails at $i=0$ whenever $k=1$, since $L_{1,1}=1<2=L_{1,0}$; this is why the lemma is stated for $i\ge1$.) The second claim follows since a product of two positive, strictly increasing sequences is strictly increasing.
\end{proof}

Lemma~\ref{lem:monotone} is what allows the extrema in the converse-H\"older and Cauchy-conversion constants of Theorems~\ref{thm:main2} and~\ref{thm:main3} to be read off directly from the endpoints $i=1$ and $i=n$, exactly as in \cite[Theorems~11--12]{DJP}.

\subsection{The general power-sum theorems of \cite{DJP}}\label{sec:prelim-djp}

We now recall, without proof, the results of \cite{DJP} that Theorems~\ref{thm:main1}-\ref{thm:main3} instantiate. Throughout this subsection, $x=(x_1,\dots,x_n)\in\mathbb{R}_+^n$ is an arbitrary positive sequence, and $S_n^{[\alpha]}(x)=\sum_{i=1}^nx_i^\alpha$ as before.

\begin{theorem}[{\cite[Theorem 6]{DJP}}]\label{thm:djp6}
	Let $p>1$, $q=p/(p-1)$.
	\begin{enumerate}[(i)]
		\item If $x_i\ge1$ for all $i$, and $u,v\in\mathbb{R}$ are such that $\alpha=u/p+v/q\ge\beta>0$, then
		\begin{align*}
		\bigl(S_n^{[u]}(x)\bigr)^{1/p}\bigl(S_n^{[v]}(x)\bigr)^{1/q}\ge S_n^{[\alpha]}(x)\ge\frac{1}{n^{\alpha/\beta-1}}\bigl(S_n^{[\beta]}(x)\bigr)^{\alpha/\beta}\ge S_n^{[\beta]}(x)\ge\bigl(S_n^{[\alpha]}(x)\bigr)^{\beta/\alpha}.
		\end{align*}
		\item If $u,v\in\mathbb{R}$ are such that $\alpha=u/p+v/q$ and $0\le\alpha<\beta$, then
		\begin{align*}
		\bigl(S_n^{[u]}(x)\bigr)^{1/p}\bigl(S_n^{[v]}(x)\bigr)^{1/q}\ge S_n^{[\alpha]}(x)\ge\bigl(S_n^{[\beta]}(x)\bigr)^{\alpha/\beta}.
	\end{align*}
	\end{enumerate}
\end{theorem}

\begin{theorem}[{\cite[Theorem 7]{DJP}}]\label{thm:djp7}
	Let $0<p<1$, $q=p/(p-1)$.
	\begin{enumerate}[(i)]
		\item If $u,v\in\mathbb{R}$ are such that $\alpha=u/p+v/q\ge\beta>0$, then
		\begin{align*}
			\bigl(S_n^{[u]}(x)\bigr)^{1/p}\bigl(S_n^{[v]}(x)\bigr)^{1/q}\le S_n^{[\alpha]}(x)\le\bigl(S_n^{[\beta]}(x)\bigr)^{\alpha/\beta}.
		\end{align*}
		\item If $x_i\ge1$ for all $i$, and $u,v\in\mathbb{R}$ are such that $\alpha=u/p+v/q$ and $0<\alpha<\beta$, then
		\begin{align*}
			\bigl(S_n^{[u]}(x)\bigr)^{1/p}\bigl(S_n^{[v]}(x)\bigr)^{1/q}\le S_n^{[\alpha]}(x)\le\frac{1}{n^{\alpha/\beta-1}}\bigl(S_n^{[\beta]}(x)\bigr)^{\alpha/\beta}\le S_n^{[\beta]}(x)\le\bigl(S_n^{[\alpha]}(x)\bigr)^{\beta/\alpha}.
		\end{align*}
	\end{enumerate}
\end{theorem}

Since the terminal exponent $\beta/\alpha$ in the displayed chain is undefined at $\alpha=0$, we state part (ii) for $0<\alpha<\beta$ throughout this paper, which is all that Theorem \ref{thm:main1}(iv) requires.
\medskip
\begin{theorem}[{\cite[Theorem 8]{DJP}}]\label{thm:djp8}
	Let $u,v\in\mathbb{R}$, $\alpha=u/p+v/q$, and suppose $0<m<M$, where 
	$$m=\min_i\{x_i^{(u-v)/p}\},\qquad M=\max_i\{x_i^{(u-v)/p}\}.$$
	\begin{enumerate}[(i)]
		\item If $p>1$, then
		\begin{align*}
			(M-m)S_n^{[u]}(x)+(mM^p-Mm^p)S_n^{[v]}(x)\le(M^p-m^p)\,S_n^{[\alpha]}(x);
		\end{align*}
		if $0<p<1$, the reversed inequality holds.
		\item If $p>1$, then
		\begin{align*}
			\bigl(S_n^{[u]}(x)\bigr)^{1/p}\bigl(S_n^{[v]}(x)\bigr)^{1/q}\le\lambda\,S_n^{[\alpha]}(x),\qquad \lambda=\bigl|M^p-m^p\bigr|\,\bigl|p(M-m)\bigr|^{-1/p}\,\bigl|q(mM^p-Mm^p)\bigr|^{-1/q}.
		\end{align*}
		If $0<p<1$, the reversed inequality holds.
	\end{enumerate}
\end{theorem}

We impose $0<m<M$ explicitly, since $m=M$ (which occurs when $n=1$, or when $x_i^{(u-v)/p}$ is constant) makes $\lambda$ undefined. 
\medskip
\begin{theorem}[{\cite[Theorem 9]{DJP}}]\label{thm:djp9}
	Let $u,v\in\mathbb{R}$, $\alpha=u/2+v/2$, and set
	\begin{align*}
	m_1=\min_i\{x_i^{u/2}\},\ M_1=\max_i\{x_i^{u/2}\},\qquad m_2=\min_i\{x_i^{v/2}\},\ M_2=\max_i\{x_i^{v/2}\}.
\end{align*}
	Then
	\begin{align*}
	1\le\frac{S_n^{[u]}(x)S_n^{[v]}(x)}{\bigl(S_n^{[\alpha]}(x)\bigr)^2} &\le\frac14\left(\sqrt{\frac{M_1M_2}{m_1m_2}}+\sqrt{\frac{m_1m_2}{M_1M_2}}\right)^2,
\\
	\frac{S_n^{[u]}(x)}{S_n^{[\alpha]}(x)}-\frac{S_n^{[\alpha]}(x)}{S_n^{[v]}(x)} &\le\left(\sqrt{\frac{M_1}{m_2}}-\sqrt{\frac{m_1}{M_2}}\right)^2,
\\
	S_n^{[u]}(x)S_n^{[v]}(x)-\bigl(S_n^{[\alpha]}(x)\bigr)^2&\le\frac{n^2}{4}(M_1M_2-m_1m_2)^2, \qquad\text{and}
\\
	S_n^{[v]}(x)+\frac{m_2M_2}{M_1m_1}S_n^{[u]}(x) &\le\left(\frac{M_2}{m_1}+\frac{m_2}{M_1}\right)S_n^{[\alpha]}(x).
\end{align*}
\end{theorem}

Theorems \ref{thm:djp6}-\ref{thm:djp7} are the source of Theorem \ref{thm:main1}, applied with $x_i=L_{k,i}$, $\beta=2$, using the first identity of Theorem \ref{thm:identities} to identify $S_n^{[2]}(L_{k,\cdot})=D_k(n)$. Theorem~\ref{thm:djp8} is the source of Theorem~\ref{thm:main2}, applied with the same $x$ and $\alpha=2$, with $m,M$ evaluated at the endpoints $i=1,n$ by Lemma~\ref{lem:monotone}. Theorem~\ref{thm:djp9} is the source of Theorem~\ref{thm:main3}, applied with $x_i=L_{k,i}L_{k,i+1}$, $\alpha=1$, using the third identity of Theorem~\ref{thm:identities} to identify $S_n^{[1]}(x)=P_k(n)$, again with $m_1,M_1,m_2,M_2$ evaluated at $i=1,n$ by the second part of Lemma~\ref{lem:monotone}. The proofs are carried out in Section \ref{sect3}.

\section{Proofs of Theorems \ref{thm:identities} to \ref{thm:main3} }\label{sect3}

\subsection{Proof of Theorem \ref{thm:identities}}

We prove the four identities in turn.

\paragraph{Sum of squares, $L_{k,\cdot}$.} By Lemma \ref{lem:step}, $L_{k,i+1}-L_{k,i-1}=kL_{k,i}$ for every $i\ge1$. Multiplying both sides by $L_{k,i}$ gives
\begin{align*}
L_{k,i}L_{k,i+1}-L_{k,i-1}L_{k,i}=kL_{k,i}^2,\qquad\text{i.e.}\qquad L_{k,i}^2=\frac{1}{k}\bigl(L_{k,i}L_{k,i+1}-L_{k,i-1}L_{k,i}\bigr).
\end{align*}
Summing over $i=1,\dots,n$, the right-hand side telescopes as
\begin{align*}
\sum_{i=1}^nL_{k,i}^2=\frac{1}{k}\sum_{i=1}^n\bigl(L_{k,i}L_{k,i+1}-L_{k,i-1}L_{k,i}\bigr)=\frac{1}{k}\bigl(L_{k,n}L_{k,n+1}-L_{k,0}L_{k,1}\bigr).
\end{align*}
Since $L_{k,0}=2$ and $L_{k,1}=k$, we have $L_{k,0}L_{k,1}=2k$, so
\begin{align*}
\sum_{i=1}^nL_{k,i}^2=\frac{L_{k,n}L_{k,n+1}-2k}{k}=D_k(n).
\end{align*}

\paragraph{Sum of squares, $F_{k,\cdot}$.} Identically, Lemma \ref{lem:step} gives $F_{k,i+1}-F_{k,i-1}=kF_{k,i}$, hence $$F_{k,i}^2=\frac1k(F_{k,i}F_{k,i+1}-F_{k,i-1}F_{k,i}),$$ and summing telescopes to
\[
\sum_{i=1}^nF_{k,i}^2=\frac1k\bigl(F_{k,n}F_{k,n+1}-F_{k,0}F_{k,1}\bigr)=\frac{F_{k,n}F_{k,n+1}}{k},
\]
using $F_{k,0}=0$, $F_{k,1}=1$.

\paragraph{Sum of products $L_{k,i}L_{k,i+1}$.} By Lemma \ref{lem:product}, for every $i\ge0$,
\begin{equation}\label{eq:pf-lem-product}
	L_{k,i}L_{k,i+1}=L_{k,2i+1}+k(-1)^i.
\end{equation}
Summing \eqref{eq:pf-lem-product} over $i=1,\dots,n$,
\begin{equation}\label{eq:pf-step1}
	\sum_{i=1}^nL_{k,i}L_{k,i+1}=\sum_{i=1}^nL_{k,2i+1}+k\sum_{i=1}^n(-1)^i.
\end{equation}
The second sum on the right is elementary: $\sum_{i=1}^n(-1)^i$ equals $-1$ if $n$ is odd and $0$ if $n$ is even, i.e.
\begin{equation}\label{eq:pf-altsum}
	\sum_{i=1}^n(-1)^i=\frac{(-1)^n-1}{2}.
\end{equation}
For the first sum on the right-hand side of \eqref{eq:pf-lem-product}, the defining recurrence $L_{k,2i+2}=kL_{k,2i+1}+L_{k,2i}$ (the case $n=2i+1$ of $L_{k,n+1}=kL_{k,n}+L_{k,n-1}$) rearranges to $$L_{k,2i+1}=\frac1k(L_{k,2i+2}-L_{k,2i}),$$ so, summing over $i=0,\dots,m$ and telescoping,
\[
\sum_{i=0}^mL_{k,2i+1}=\frac1k\sum_{i=0}^m\bigl(L_{k,2i+2}-L_{k,2i}\bigr)=\frac1k\bigl(L_{k,2m+2}-L_{k,0}\bigr)=\frac{L_{k,2m+2}-2}{k}.
\]
Taking $m=n$ and subtracting the $i=0$ term $L_{k,1}=k$ gives
\begin{equation}\label{eq:pf-step2}
	\sum_{i=1}^nL_{k,2i+1}=\frac{L_{k,2n+2}-2}{k}-k.
\end{equation}
By Lemma~\ref{lem:squares} with $m=n+1$, $L_{k,n+1}^2=L_{k,2n+2}+2(-1)^{n+1}$, so
\begin{equation}\label{eq:pf-step3}
	L_{k,2n+2}=L_{k,n+1}^2-2(-1)^{n+1}=L_{k,n+1}^2+2(-1)^n.
\end{equation}
Substituting \eqref{eq:pf-step3} into \eqref{eq:pf-step2}, and then \eqref{eq:pf-step2} and \eqref{eq:pf-altsum} into \eqref{eq:pf-step1},
\begin{align*}
\sum_{i=1}^nL_{k,i}L_{k,i+1}=\frac{L_{k,n+1}^2+2(-1)^n-2}{k}-k+k\cdot\frac{(-1)^n-1}{2}.
\end{align*}
Expanding and regrouping the constant and $(-1)^n$ terms,
\begin{align*}
\sum_{i=1}^nL_{k,i}L_{k,i+1}&=\frac{L_{k,n+1}^2}{k}-k+\left(\frac{2}{k}+\frac{k}{2}\right)(-1)^n-\left(\frac2k+\frac k2\right)\\
&=\frac{L_{k,n+1}^2}{k}-k+\left((-1)^n-1\right)\left(\frac2k+\frac k2\right)\\
&=P_k(n).
\end{align*}

\paragraph{Cross-identity.} By the Binet formulas of Section~\ref{sec:background}, for every $i\ge1$,
\begin{align*}
F_{k,i}L_{k,i}=\frac{\sigma_1^i-\sigma_2^i}{\sigma_1-\sigma_2}\cdot(\sigma_1^i+\sigma_2^i)=\frac{\sigma_1^{2i}-\sigma_2^{2i}}{\sigma_1-\sigma_2}=F_{k,2i}.
\end{align*}
\paragraph{Reduction at $k=1$.} At $k=1$, $D_1(n)=L_nL_{n+1}-2$, the classical identity. For $P_k(n)$, at $k=1$,
\begin{align*}
P_1(n)=L_{n+1}^2-1+\bigl((-1)^n-1\bigr)\Bigl(2+\frac12\Bigr)=L_{n+1}^2-1+\frac52\bigl((-1)^n-1\bigr)=L_{n+1}^2+\frac{5(-1)^n-7}{2},
\end{align*}
the classical ordinary-Lucas identity. This completes the proof of Theorem \ref{thm:identities}. \qed

\subsection{Proof of Theorem \ref{thm:main1}}

We apply Theorems \ref{thm:djp6} and \ref{thm:djp7} with $x_i=L_{k,i}$, $i=1,\dots,n$, and $\beta=2$.

We first check the hypothesis $x_i\ge1$, which is required by Theorem \ref{thm:djp6}(i) and Theorem \ref{thm:djp7}(ii). Since $L_{k,1}=k\ge1$ and, by Lemma \ref{lem:monotone}, $(L_{k,i})_{i\ge1}$ is strictly increasing, we have $L_{k,i}\ge L_{k,1}=k\ge1$ for every $i\ge1$. Thus the hypothesis holds throughout, and in particular it is available (though not needed) in the two parts of Theorems \ref{thm:djp6} and \ref{thm:djp7} that do not require it.

By Theorem \ref{thm:identities}, $S_n^{[2]}(L_{k,\cdot})=\sum_{i=1}^nL_{k,i}^2=D_k(n)$.

\begin{enumerate}[(i)]
	\item Theorem \ref{thm:djp6}(i), with $\beta=2$ and $\alpha=u/p+v/q\ge2$, gives
	\begin{align*}
	\bigl(S_n^{[u]}\bigr)^{1/p}\bigl(S_n^{[v]}\bigr)^{1/q}\ge S_n^{[\alpha]}\ge\frac1{n^{\alpha/2-1}}\bigl(S_n^{[2]}\bigr)^{\alpha/2}\ge S_n^{[2]}\ge\bigl(S_n^{[\alpha]}\bigr)^{2/\alpha}.
	\end{align*}
	Substituting $S_n^{[2]}=D_k(n)$ gives Theorem \ref{thm:main1}(i).
	
	\item Theorem \ref{thm:djp6}(ii), with $\beta=2$ and $0\le\alpha<2$, gives $(S_n^{[u]})^{1/p}(S_n^{[v]})^{1/q}\ge S_n^{[\alpha]}\ge (S_n^{[2]})^{\alpha/2}$. Substituting $S_n^{[2]}=D_k(n)$ gives Theorem \ref{thm:main1}(ii).
	
	\item Theorem \ref{thm:djp7}(i), with $0<p<1$ and $\alpha=u/p+v/q\ge2$, gives $$(S_n^{[u]})^{1/p}(S_n^{[v]})^{1/q}\le S_n^{[\alpha]}\le (S_n^{[2]})^{\alpha/2}.$$ 
	Note that this part of Theorem \ref{thm:djp7} does not require $x_i\ge1$. Substituting $S_n^{[2]}=D_k(n)$ gives Theorem \ref{thm:main1}(iii).
	
\item Theorem \ref{thm:djp7}(ii), with $x_i\ge1$ (verified above), $0<p<1$, $\beta=2$, and $0<\alpha<2$, gives
\begin{align*}
	\bigl(S_n^{[u]}\bigr)^{1/p}\bigl(S_n^{[v]}\bigr)^{1/q}\le S_n^{[\alpha]}\le\frac1{n^{\alpha/2-1}}\bigl(S_n^{[2]}\bigr)^{\alpha/2}\le S_n^{[2]}\le\bigl(S_n^{[\alpha]}\bigr)^{2/\alpha}.
\end{align*}
Substituting $S_n^{[2]}=D_k(n)$ gives Theorem \ref{thm:main1}(iv) for $0<\alpha<2$; the endpoint $\alpha=0$ is excluded because the terminal exponent $2/\alpha$ is undefined there, exactly as already noted in the statement of Theorem \ref{thm:main1}. \qed
\end{enumerate}

\subsection{Proof of Theorem \ref{thm:main2}}

We apply Theorem \ref{thm:djp8} with $x_i=L_{k,i}$ and $\alpha=u/p+v/q=2$, so that $S_n^{[\alpha]}(x)=S_n^{[2]}(x)=D_k(n)$ by Theorem \ref{thm:identities}.

\paragraph{The hypothesis $0<m<M$.} Write $r=(u-v)/p$; since $u\ne v$ and $p\ne0$, $r\ne0$. The map $t\mapsto t^r$ is then strictly monotonic (increasing if $r>0$, decreasing if $r<0$) on $(0,\infty)$. By Lemma~\ref{lem:monotone}, $i\mapsto L_{k,i}$ is strictly increasing on $\{1,\dots,n\}$; composing with $t\mapsto t^r$, the sequence $i\mapsto L_{k,i}^r$ is strictly monotonic on $\{1,\dots,n\}$, and since $n\ge2$ this sequence takes at least two distinct values. A strictly monotonic sequence on a finite index set attains its minimum and maximum at the two endpoints $i=1,n$, and these two values are distinct. Hence
\[
m=\min\bigl\{L_{k,1}^r,L_{k,n}^r\bigr\}<\max\bigl\{L_{k,1}^r,L_{k,n}^r\bigr\}=M,
\]
which is exactly the definition of $m,M$ in Theorem \ref{thm:main2}, with $0<m<M$ as required by Theorem~\ref{thm:djp8}.

\paragraph{Derivation of the $\lambda$-inequality.} We now derive the second inequality of Theorem~\ref{thm:djp8}(ii) from the first, rather than simply citing it, since the constant $\lambda$ requires care. Set $A=M-m>0$ and $B=mM^p-Mm^p=mM(M^{p-1}-m^{p-1})>0$ (positive since $p>1$ and $0<m<M$). These are exactly the coefficients in Theorem~\ref{thm:djp8}(i). For $p>1$ and $q=p/(p-1)$, weighted AM-GM ($\tfrac1p+\tfrac1q=1$) gives, for any $a,b>0$, $\tfrac{a}{p}+\tfrac{b}{q}\ge a^{1/p}b^{1/q}$. Taking $a=pA\,S_n^{[u]}(x)$ and $b=qB\,S_n^{[v]}(x)$,
\[
A\,S_n^{[u]}(x)+B\,S_n^{[v]}(x)\ \ge\ (pA)^{1/p}(qB)^{1/q}\bigl(S_n^{[u]}(x)\bigr)^{1/p}\bigl(S_n^{[v]}(x)\bigr)^{1/q}.
\]
Combining with Theorem~\ref{thm:djp8}(i), $A\,S_n^{[u]}(x)+B\,S_n^{[v]}(x)\le(M^p-m^p)S_n^{[\alpha]}(x)$, and rearranging,
\[
\bigl(S_n^{[u]}(x)\bigr)^{1/p}\bigl(S_n^{[v]}(x)\bigr)^{1/q}\ \le\ \frac{M^p-m^p}{(pA)^{1/p}(qB)^{1/q}}\,S_n^{[\alpha]}(x)=\lambda\,S_n^{[\alpha]}(x),
\]
with $\lambda=|M^p-m^p|\,|p(M-m)|^{-1/p}\,|q(mM^p-Mm^p)|^{-1/q}$ exactly as stated (the absolute values are redundant here since $M>m>0$, but keep the statement uniform with the $0<p<1$ case). 
Substituting $m,M$ and $S_n^{[\alpha]}(x)=D_k(n)$ into Theorem \ref{thm:djp8}(i)-(ii) (for $p>1$; the case $0<p<1$ reverses both inequalities directly by the corresponding clause of Theorem \ref{thm:djp8}) gives precisely the two inequalities of Theorem \ref{thm:main2}. \qed

\subsection{Proof of Theorem \ref{thm:main3}}

We apply Theorem \ref{thm:djp9} with $x_i=L_{k,i}L_{k,i+1}$, $i=1,\dots,n$, and $u,v$ with $u+v=2$, so that $\alpha=u/2+v/2=1$ and, by Theorem \ref{thm:identities}, $S_n^{[\alpha]}(x)=S_n^{[1]}(x)=\sum_{i=1}^nx_i=\sum_{i=1}^nL_{k,i}L_{k,i+1}=P_k(n)$.

By Lemma \ref{lem:monotone}, $(x_i)_{i\ge1}=(L_{k,i}L_{k,i+1})_{i\ge1}$ is strictly increasing. Exactly as in the proof of Theorem \ref{thm:main2}, for any fixed real exponent $r$ the sequence $i\mapsto x_i^r$ is monotonic on $\{1,\dots,n\}$ and so attains its extrema at $i=1,n$; applying this with $r=u/2$ and $r=v/2$ gives
\begin{align*}
m_1&=\min\{x_1^{u/2},x_n^{u/2}\},\qquad M_1=\max\{x_1^{u/2},x_n^{u/2}\},\\
m_2&=\min\{x_1^{v/2},x_n^{v/2}\},\qquad M_2=\max\{x_1^{v/2},x_n^{v/2}\},
\end{align*}
matching the constants in Theorem \ref{thm:main3}. Substituting these and $S_n^{[1]}(x)=P_k(n)$ into the four inequalities of Theorem \ref{thm:djp9} gives precisely the four inequalities of Theorem \ref{thm:main3}. \qed

\section{Cross-Family Bound and Further Identities}\label{sect4}

\subsection{Proof of Theorem~\ref{thm:main4}}

Applying the Cauchy-Schwarz inequality (the case $p=q=2$ of H\"older's inequality, recalled in Section \ref{sec:background}) to the two sequences $(F_{k,i})_{i=1}^n$ and $(L_{k,i})_{i=1}^n$, we get
\[
\Bigl(\sum_{i=1}^nF_{k,i}L_{k,i}\Bigr)^2\le\Bigl(\sum_{i=1}^nF_{k,i}^2\Bigr)\Bigl(\sum_{i=1}^nL_{k,i}^2\Bigr).
\]
By the cross-identity of Theorem~\ref{thm:identities}, $F_{k,i}L_{k,i}=F_{k,2i}$ for every $i\ge1$, so the left-hand side equals $\bigl(\sum_{i=1}^nF_{k,2i}\bigr)^2$. By the first two identities of Theorem~\ref{thm:identities}, $\sum_{i=1}^nF_{k,i}^2=F_{k,n}F_{k,n+1}/k$ and $\sum_{i=1}^nL_{k,i}^2=D_k(n)$. Substituting gives
\begin{align*}
\Bigl(\sum_{i=1}^nF_{k,2i}\Bigr)^2\le\frac{F_{k,n}F_{k,n+1}}{k}\cdot D_k(n),
\end{align*}
which is the inequality of Theorem~\ref{thm:main4}.

\paragraph{Equality case.} Equality in Cauchy-Schwarz holds if and only if $(F_{k,1},\dots,F_{k,n})$ and $(L_{k,1},\dots,L_{k,n})$ are proportional. For $n=1$ this holds trivially (a single pair of numbers is always proportional), giving equality. For $n\ge2$, proportionality would in particular force $F_{k,1}/L_{k,1}=F_{k,2}/L_{k,2}$. Since $F_{k,1}=1$, $L_{k,1}=k$, $F_{k,2}=k$, and $L_{k,2}=kL_{k,1}+L_{k,0}=k^2+2$, this equality reads
\begin{align*}
\frac1k=\frac{k}{k^2+2},
\end{align*}
i.e.\ $k^2+2=k^2$, which is false for every $k\ge1$. Hence the two sequences are never proportional when $n\ge2$, and the inequality is strict. This proves equality holds if and only if $n=1$. \qed

\subsection{Further $k$-Fibonacci and $k$-Lucas identities}

The identities of Theorem \ref{thm:identities} are not the only closed forms available for $k$-Fibonacci/$k$-Lucas power sums at $\beta=1$. We record six more, all provable by the same reindex-and-subtract technique used throughout Section \ref{sect3}, and all reducing at $k=1$ to classical Fibonacci/Lucas sum identities.

\begin{theorem}\label{thm:further}
	For every positive integer $k$ and every $n\ge1$,
	\begin{align}
		\sum_{i=1}^nF_{k,i}&=\frac{F_{k,n}+F_{k,n+1}-1}{k}, \label{eq:sumF}\\
		\sum_{i=1}^nL_{k,i}&=\frac{L_{k,n}+L_{k,n+1}-k-2}{k}, \label{eq:sumL}\\
		\sum_{i=1}^nF_{k,2i-1}&=\frac{F_{k,2n}}{k}, \label{eq:sumFodd}\\
		\sum_{i=1}^nF_{k,2i}&=\frac{F_{k,2n+1}-1}{k}, \label{eq:sumFeven}\\
		\sum_{i=1}^nL_{k,2i-1}&=\frac{L_{k,2n}-2}{k}, \label{eq:sumLodd}\\
		\sum_{i=1}^nL_{k,2i}&=\frac{L_{k,2n+1}-k}{k}. \label{eq:sumLeven}
	\end{align}
	At $k=1$, \eqref{eq:sumF}--\eqref{eq:sumLeven} reduce respectively to the classical identities $\sum_{i=1}^nF_i=F_{n+2}-1$, $\sum_{i=1}^nL_i=L_{n+2}-3$, $\sum_{i=1}^nF_{2i-1}=F_{2n}$, $\sum_{i=1}^nF_{2i}=F_{2n+1}-1$, $\sum_{i=1}^nL_{2i-1}=L_{2n}-2$, and $\sum_{i=1}^nL_{2i}=L_{2n+1}-1$.
\end{theorem}

\begin{proof}
	\emph{Identity \eqref{eq:sumF}.} By Lemma \ref{lem:step}, $F_{k,i+1}-F_{k,i-1}=kF_{k,i}$ for every $i\ge1$. Summing over $i=1,\dots,n$,
	\begin{align*}
	k\sum_{i=1}^nF_{k,i}=\sum_{i=1}^nF_{k,i+1}-\sum_{i=1}^nF_{k,i-1}=\sum_{j=2}^{n+1}F_{k,j}-\sum_{j=0}^{n-1}F_{k,j}.
	\end{align*}
	Writing $T=\sum_{i=1}^nF_{k,i}$, the first sum on the right is $T-F_{k,1}+F_{k,n+1}=T-1+F_{k,n+1}$ and the second is $F_{k,0}+T-F_{k,n}=T-F_{k,n}$ (using $F_{k,0}=0$). Subtracting,
	\begin{align*}
	k T=\bigl(T-1+F_{k,n+1}\bigr)-\bigl(T-F_{k,n}\bigr)=F_{k,n}+F_{k,n+1}-1,
\end{align*}
	which gives \eqref{eq:sumF}.
	
	\emph{Identity \eqref{eq:sumL}.} Identically, with $L_{k,i+1}-L_{k,i-1}=kL_{k,i}$ and $L_{k,0}=2$, $L_{k,1}=k$, the same reindex-and-subtract argument gives $kT_L=L_{k,n}+L_{k,n+1}-L_{k,0}-L_{k,1}=L_{k,n}+L_{k,n+1}-k-2$, where $T_L=\sum_{i=1}^nL_{k,i}$, proving \eqref{eq:sumL}.
	
	\emph{Identities \eqref{eq:sumFodd}--\eqref{eq:sumLeven}.} These follow from Lemma~\ref{lem:step} applied at even or odd index, exactly as in the proof of the third identity of Theorem~\ref{thm:identities}. Applying Lemma~\ref{lem:step} with $i\mapsto2i-1$ gives $F_{k,2i}-F_{k,2i-2}=kF_{k,2i-1}$ for every $i\ge1$; summing over $i=1,\dots,n$ telescopes directly to
\begin{align*}
	k\sum_{i=1}^nF_{k,2i-1}=\sum_{i=1}^n\bigl(F_{k,2i}-F_{k,2i-2}\bigr)=F_{k,2n}-F_{k,0}=F_{k,2n},
\end{align*}
	which gives \eqref{eq:sumFodd}. Applying Lemma~\ref{lem:step} with $i\mapsto2i$ gives $F_{k,2i+1}-F_{k,2i-1}=kF_{k,2i}$; summing and telescoping,
	\begin{align*}
	k\sum_{i=1}^nF_{k,2i}=F_{k,2n+1}-F_{k,1}=F_{k,2n+1}-1,
\end{align*}
	which gives \eqref{eq:sumFeven}. The same two substitutions applied to $L_{k,i+1}-L_{k,i-1}=kL_{k,i}$ give
	\begin{align*}
	k\sum_{i=1}^nL_{k,2i-1}=L_{k,2n}-L_{k,0}=L_{k,2n}-2,\qquad k\sum_{i=1}^nL_{k,2i}=L_{k,2n+1}-L_{k,1}=L_{k,2n+1}-k,
\end{align*}
	which give \eqref{eq:sumLodd} and \eqref{eq:sumLeven} respectively.
	
	The reductions at $k=1$ follow by direct substitution, using $F_{1,n+2}=F_{1,n}+F_{1,n+1}$ and $L_{1,n+2}=L_{1,n}+L_{1,n+1}$ to rewrite \eqref{eq:sumF}--\eqref{eq:sumL} in the stated classical form.
\end{proof}

Identity \eqref{eq:sumFeven} makes the cross-family bound of Theorem~\ref{thm:main4} fully explicit: the left-hand side is no longer an unevaluated sum but a closed form in $k,n$.

\begin{cor}\label{cor:explicit}
	For every $k\ge1$ and $n\ge1$,
	\[
	\left(\frac{F_{k,2n+1}-1}{k}\right)^2\le\frac{F_{k,n}F_{k,n+1}}{k}\cdot D_k(n).
	\]
\end{cor}
\begin{proof}
Substitute \eqref{eq:sumFeven} into Theorem \ref{thm:main4}.
\end{proof}

\section{Concluding Remarks}\label{sect5}

We have shown that the H\"older/Cauchy power-sum machinery of Dujella, Jak\v seti\'c and Pe\v cari\'c \cite{DJP}, originally developed for the Fibonacci sequence alone, extends uniformly across the one-parameter family of $k$-Fibonacci and $k$-Lucas numbers of Falc\'on and Plaza. The extension rests on four load-bearing identities (Theorem \ref{thm:identities}), six further linear power-sum identities (Theorem \ref{thm:further}), and a monotonicity lemma (Lemma~\ref{lem:monotone}) that lets the converse-H\"older and Cauchy-conversion constants be read off at the endpoints $i=1,n$ for every $k$ simultaneously, rather than being re-derived case by case.

\subsection{Specializations at $k=1$}

Since $D_1(n)=L_nL_{n+1}-2$ and $P_1(n)=L_{n+1}^2+\tfrac{5(-1)^n-7}{2}$ (Theorem \ref{thm:identities}), setting $k=1$ in Theorems \ref{thm:main1}-\ref{thm:main3} recovers, verbatim, the ordinary-Lucas analogues of \cite[Theorems 10-12]{DJP}; the same four-part H\"older sandwich, converse-H\"older refinement, and Cauchy-conversion refinement that \cite{DJP} established for the Fibonacci sequence, now proved instead for the Lucas sequence, with $L_nL_{n+1}-2$ playing the role that $F_nF_{n+1}$ plays in \cite{DJP}. To the best of our knowledge these Lucas-sequence chains have not previously been recorded in this form. The work in  \cite{DJP} treats only the Fibonacci case, and the Lucas case does not follow from it by any direct substitution, since the identities $\sum L_i^2=L_nL_{n+1}-2$ and $\sum L_iL_{i+1}=L_{n+1}^2+\tfrac{5(-1)^n-7}{2}$ are not simply the Fibonacci identities \eqref{eq:fib-identities} with $F$ relabeled as $L$.

The cross-family bound specializes to a fully explicit, elementary inequality relating ordinary Fibonacci and Lucas numbers.

\medskip

\begin{cor}\label{cor:k1cross}
	For every $n\ge1$,
	\[
	(F_{2n+1}-1)^2\le F_nF_{n+1}\bigl(L_nL_{n+1}-2\bigr),
	\]
	with equality if and only if $n=1$.
\end{cor}
\begin{proof}
	Set $k=1$ in Corollary~\ref{cor:explicit}, using $D_1(n)=L_nL_{n+1}-2$; the equality case is the $k=1$ instance of Theorem~\ref{thm:main4}.
\end{proof}

For $n\ge2$ the inequality is strict; for example, the ratio of the two sides is $1.25$ at $n=2$, $1.033$ at $n=4$, and $1.001$ at $n=8$. We do not prove that this ratio is monotonic or that it tends to $1$; the examples above only illustrate that it is close to $1$ and small at these values of $n$.

\subsection{A second instance: $k=2$}

The family is not merely a formal device. At $k=2$, $F_{2,\cdot}$ and $L_{2,\cdot}$ are the classical Pell and Pell-Lucas sequences, $0,1,2,5,12,29,70,\dots$ and $2,2,6,14,34,82,198,\dots$, and Theorems~\ref{thm:main1}-\ref{thm:main4} apply to them exactly as stated, with $D_2(n)=\tfrac{L_{2,n}L_{2,n+1}-4}{2}$ in place of $D_k(n)$. Every inequality proved in this paper is therefore simultaneously a statement about Fibonacci and Lucas numbers ($k=1$), Pell and Pell--Lucas numbers ($k=2$), and every further $k$-Fibonacci/$k$-Lucas pair, from a single unified proof.

\subsection{Further directions}

Two natural extensions remain open. First, the further identities of Theorem~\ref{thm:further} are linear ($\beta=1$) sums; the original Fibonacci paper \cite{DJP} lists several higher-order closed forms (weighted binomial sums, products such as $F_iF_{3i}$) that were left unused even in the Fibonacci case, and it would be of interest to determine which of these admit $k$-Fibonacci/$k$-Lucas analogues by the same telescoping technique. Second, the $k$-Fibonacci and $k$-Lucas numbers are themselves a special case ($P=k,\,Q=-1$) of the general Horadam/Lucas sequences $U_n(P,Q)$, $V_n(P,Q)$ satisfying $U_n=PU_{n-1}-QU_{n-2}$. The step identity of Lemma \ref{lem:step} extends immediately to this setting, but the monotonicity of Lemma \ref{lem:monotone} does not: our proof uses $L_{k,i}>0$ for all $i\ge0$, which holds because $k\ge1$, and need not hold for general $P,Q$ or general initial conditions, where terms can change sign or fail to be monotonic. It would be of interest to investigate whether analogues of Theorems \ref{thm:identities}-\ref{thm:main4} extend to Horadam sequences under explicit conditions on $P,Q$ and the initial terms that guarantee positivity and monotonicity of the resulting sequences, which would subsume the present results at $(P,Q)=(k,-1)$.

\section*{Acknowledgments}
The first author is deeply grateful to his supervisor, the late Professor Florian Luca, whose foundational contributions to the study of linear recurrences shaped the direction of this work, and to whose memory this paper is dedicated. The first author also thanks the Mathematics Division of Stellenbosch University for funding his PhD studies.

\section*{Disclosure Statement}
The corresponding author reports there are no competing interests to declare.

\section*{Funding}
The first author's doctoral studies were funded by the Mathematics Division of Stellenbosch University and no specific grant was awarded for this research.

	\section*{Addresses}
	$ ^{1} $ Mathematics Division, Stellenbosch University, Stellenbosch, South Africa.
	
	Email: \url{hbatte91@gmail.com}
	
	$ ^{2} $ Department of Mathematics, Makerere University, Kampala, Uganda.
	
	Email: \url{kaggwaprosper58@gmail.com}
\end{document}